\documentclass[11pt]{article}
\title{\vskip-1.0em\sf Realization of compact spaces as cb-Helson sets}
\author{\sc Yemon Choi}
\date{29th April 2015}

\RequirePackage{amsmath,amssymb,amsfonts}
\RequirePackage{amsthm}   %Inbuilt environments for defns etc
\newcounter{pulse}[section]
\numberwithin{pulse}{section}  % all thms, etc set from pulse

\newcommand{\thf}{} % alias for theorem heading font (in personal style)

\theoremstyle{plain}
\newtheorem{thm}[pulse]{\thf Theorem}

\newtheorem{prop}[pulse]{\thf Proposition}
\newtheorem{lem}[pulse]{\thf Lemma}

\theoremstyle{definition}
\newtheorem{dfn}[pulse]{\thf Definition}

\theoremstyle{remark}
\newtheorem{rem}[pulse]{\thf Remark}

\newtheorem*{qnsstar}{\thf Questions}

\renewcommand{\emph}[1]{{\sl #1}\/} % use slanted font for emphasis
\newcommand{\st}{\mathbin{\colon}}
\newcommand{\defeq}{:=}
\newcommand{\dt}[1]{{\it#1}\/}
\newcommand{\iso}{\cong}

\newcommand{\tp}{\mathbin{\otimes}}

\newenvironment{YCnum}{%
\begin{enumerate}

}{\end{enumerate}\ignorespacesafterend}

\newcommand{\wstar}{\ensuremath{{\rm w}^*}}

\newcommand{\pair}[2]{\langle{#1},\,{#2}\rangle}

\newcommand{\twomat}[4]{\left[\begin{matrix} #1 & #2 \\ #3 & #4 \end{matrix} \right] }

\newcommand{\ptp}{\mathbin{\widehat{\otimes}}} % projective tensor product
\newcommand{\norm}[1]{\Vert{#1}\Vert}

\newcommand{\Bdd}{{\mathcal B}}

\newcommand{\Cplx}{\mathbb C}
\newcommand{\Nat}{\mathbb N}

\newcommand{\Real}{\mathbb R}

\newcommand{\VN}{\operatorname{VN}}
\newcommand{\FA}{\operatorname{A}}
\newcommand{\Cst}{\operatorname{C}^*}

\newcommand{\wswscts}{\wstar-\wstar\ continuous}

\newcommand{\ball}{\mathop{\sf ball}}
\newcommand{\cbhel}{\operatorname{Hel}\nolimits_{\rm cb}}
\newcommand{\actson}{\curvearrowright}
\newcommand{\Ind}{{\mathbb I}}
\newcommand{\Mat}{{\mathbb M}}

\newcommand{\bbG}{{\mathbb G}}

\newcommand{\bbU}{{\mathbb U}}

\newcommand{\cH}{\mathcal{H}}
\newcommand{\cK}{\mathcal{K}}
\newcommand{\cM}{\mathcal{M}}
\newcommand{\cU}{\mathcal{U}}

\newcommand{\bj}{{\bf j}}
\newcommand{\bu}{{\bf u}}
\newcommand{\bv}{{\bf v}}
\newcommand{\bx}{{\bf x}}
\newcommand{\bU}{{\bf U}}
\newcommand{\bS}{{\bf S}}
\newcommand{\bmu}{{\boldsymbol{\mu}}}

\newcommand{\Gm}{\Gamma}
\newcommand{\Om}{\Omega}
\newcommand{\gm}{\gamma}
\newcommand{\om}{\omega}
\newcommand{\lm}{\lambda}

\newcommand{\dist}{\operatorname{dist}\nolimits}

\newcommand{\ip}[2]{(#1\mid#2)}

\newcommand{\id}{\operatorname{\mathsf{id}}}
\newcommand{\Tr}{\operatorname{Tr}}

\newcommand{\ev}{\operatorname{ev}}

\newcommand{\sV}{{\sf V}}

\newcommand{\bbF}{{\mathbb F}}
\newcommand{\R}{{\mathbb R}}
\newcommand{\T}{{\mathbb T}}
\newcommand{\MAX}{\operatorname{\sf max}}
\newcommand{\MIN}{\operatorname{\sf min}}
\newcommand{\inc}{\operatorname{\sf inc}\nolimits}

\newcommand{\SUTWO}{{\rm SU(2,\Cplx)}}

\newcommand{\bprod}{\prod\nolimits^{\sf B}} % ad hoc notation for the Banach space direct product

\renewcommand{\thf}{\sc} % alias for theorem heading font

\usepackage[usenames]{color}
\renewcommand{\dt}[1]{\textcolor{Bittersweet}{\textsf{#1}}}
\usepackage{sectsty}
\allsectionsfont{\sffamily}

\begin{document}
\maketitle

{\it Dedicated to Anthony To-Ming Lau, with thanks for all his work on behalf of the international community in Abstract Harmonic Analysis.}

\begin{abstract}
We show that, given a compact Hausdorff space $\Omega$, there is a compact group $\bbG$ and a homeomorphic embedding of $\Omega$ into $\bbG$, such that the restriction map $\FA(\bbG)\to C(\Omega)$ is a complete quotient map of operator spaces. In particular, this shows that there exist compact groups which contain infinite cb-Helson subsets, answering a question raised in~\cite{CS_AGquot}.
%Choi--Samei, Proc.~AMS 2013.
 A negative result from \cite{CS_AGquot} is also improved.

\medskip\noindent
MSC 2010 classification: Primary 43A30; Secondary 46L07\hfill\break
\end{abstract}

%\tableofcontents

%\vfill\eject

\begin{section}{Introduction}
The notion of a cb-Helson subset of a locally compact group was introduced in~\cite{CS_AGquot}, in connection with the study of quotients of Fourier algebras. For instance, the following result can be obtained by an easy modification of the proof of \cite[Theorem~B]{CS_AGquot}.

\begin{thm}[Corollary of work in \cite{CS_AGquot}]
Let $G$ be a {\rm SIN} group and let $J$ be a closed ideal in the Fourier algebra $\FA(G)$. Suppose that $\FA(G)/J$ is completely boundedly isomorphic to a closed subalgebra of $\Bdd(\cH)$ for some Hilbert space $\cH$. Then there is a cb-Helson subset $E\subset G$ such that $J = \{ f \in \FA(G) \colon f\vert_E=0\}$.
\end{thm}

We defer the definition of a cb-Helson set to Section~\ref{s:prelim}. For now, we note that such sets appear to be rather hard to come by: finite subsets of locally compact groups have the cb-Helson property, but hitherto no infinite cb-Helson sets were known to exist.
Uncountable Helson sets can be found in $\T$ and $\R$, using classical constructions of Fourier analysis (see e.g.~\cite[Theorem 5.6.6]{Rudin_FAG}). By Herz's restriction theorem, these give uncountable Helson sets in any locally compact group that contains a closed copy of either $\T$ or $\R$. In contrast, every cb-Helson subset of a locally compact abelian group must be finite (see \cite[Theorem~C]{CS_AGquot}).

The purpose of this note is twofold. Firstly, we strengthen \cite[Theorem C]{CS_AGquot} as follows:

\begin{thm}\label{t:amen-as-discrete}
Let $G$ be a locally compact group, and $G_d$ the same group equipped with the discrete topology. Suppose $G_d$ is amenable. Then every cb-Helson subset of $G$ is finite.
\end{thm}

\noindent
(This applies, in particular, to any locally compact $G$ which contains a solvable subgroup of finite index.)
Secondly, answering a question raised in \cite{CS_AGquot}, we show that infinite cb-Helson sets do exist, albeit inside some ``pathologically large'' groups.

\begin{thm}\label{t:realization}
Let $\Omega$ be a compact Hausdorff space. Then there is a compact, connected group $\bbG$ and a cb-Helson subset of $\bbG$ that is homeomorphic to $\Omega$.
\end{thm}

Theorem~\ref{t:realization} can be thought of as an analogue of a folklore result for ``classical'' Helson sets, which roughly speaking says that each compact Hausdorff space $\Om$ can be homeomorphically embedded as a Helson set inside \emph{some} compact abelian group. Note that if we restrict attention to particular compact abelian groups then there can be topological restrictions on the possible Helson subsets: for instance, a Helson subset of $\T$ cannot contain any closed subset homeomorphic to $[0,1]$ (cf.~\cite[Remark 5.6.8]{Rudin_FAG}).

One drawback of Theorem~\ref{t:realization} is that we create groups specially for the purpose of containing certain spaces as cb-Helson subsets, rather than finding ``natural'' examples of groups that contain infinite cb-Helson subsets.
In particular, the following questions remain open (the second of which was suggested by the referee):

\begin{qnsstar}\
\begin{enumerate}
\item Is there a connected, linear Lie group that contains an infinite cb-Helson subset?
\item Does there exist an infinite index set $\Ind$ and some $n\in\Nat$ such that $\prod_{i\in\Ind} \bbU_n$ contains an infinite cb-Helson subset?
\end{enumerate}
\end{qnsstar}
Let us say something about the proofs. Theorem~\ref{t:amen-as-discrete} is a quick consequence of some hard (but standard) results in the theory of operator spaces; its proof will be given in Section~\ref{s:cb-Helson constants}. In contrast, the proof of Theorem~\ref{t:realization} is much longer, but only uses basic tools. The arguments are not difficult once one thinks of the right idea; the key part is based on the standard embedding of a dual operator space inside a product of matrix algebras. We will actually prove a more precise statement, given as Theorem~\ref{t:embedding theorem}.
\end{section}

\begin{section}{Notation and preliminaries}\label{s:prelim}

\subsection{Notation and terminology}
If $X$ is a Banach space then $\ball(X)$ denotes its \emph{closed} unit ball. If $B$ is a unital $\Cst$-algebra then $\bU(B)$ denotes its unitary group. If $\cH$ is a Hilbert space then we write $\cU(\cH)$ instead of $\bU(\Bdd(\cH))$.
We use $\bbU_n$ as an abbreviation for $\cU(\Cplx^n)$, i.e.~the usual group of $n\times n$ unitary matrices. Similarly, we write $\Mat_n$ as an abbreviation for $M_n(\Cplx)$.

Given a family $(X_i)_{i\in\Ind}$ of Banach spaces, we shall write $\bprod_{i\in\Ind} X_i$ for the Banach space direct product of this family, i.e.~what is sometimes called the $\ell_\infty$-direct sum. This is just to avoid ambiguity since we shall also be dealing with sets or groups that arise as Cartesian products (in the usual sense) of sets or groups.

All topological groups are assumed to be {\it a priori} Hausdorff; if $G$ is a topological group, $G_d$ denotes the same group equipped with the discrete topology.

The abbreviation WOT stands for ``weak operator topology'' (on $\Bdd(\cH)$ for some given Hilbert space $\cH$).
 Given  a topological group $G$ and a Hilbert space~$\cH$, a representation $\pi:G\actson\cH$ is said to be \dt{unitary} if $\pi(G)\subseteq \cU(\cH)$, and said to be \dt{WOT-continuous} if $\pi:G\to (\Bdd(\cH),{\rm WOT})$ is continuous as a map between topological spaces. For most of what we do $G$ is locally compact, but it is useful to have terminology which is well-defined without having to establish local compactness.

Given a family $(\sigma_i)_{i\in\Ind}$ of unitary representations of a common group $G$, with $\sigma_i:G\actson\cH_i$, we define the \dt{direct product} of this family (sometimes called the \dt{direct sum}) as follows. Let $\cH\defeq \hbox{$\ell^2$-}\bigoplus_{i\in\Ind}\cH_i$ be the $\ell^2$-direct sum of the representation spaces. Given $\bx=(x_i)_{i\in\Ind}\in G$ and ${\boldsymbol\xi}=(\xi_i)_{i\in\Ind}\in\cH$, let
\[
[ \sigma(\bx)(\boldsymbol{\xi})]_i \defeq \sigma_i(x_i)(\xi_i) \qquad(i\in\Ind).
\]
Clearly $\sigma$ is also a unitary representation: we often denote it by $\prod_{i\in\Ind}\sigma_i$.

\begin{rem}
\label{r:product-of-WOT-cts}
Let $G$ be a topological group. Then the direct product of a family of WOT-con\-tin\-uous unitary representations of $G$ is itself WOT-continuous.
In the locally compact setting, this is often deduced as an application of the correspondence between WOT-continuous representations of a locally compact group $G$ and non-degenerate $*$-representations of $L^1(G)$. However, it is not hard to give a proof which works for arbitrary topological groups; since we did not see this proof in the sources we consulted, it is included in the appendix for the reader's interest.
\end{rem}
%\YCnote{Is there an explicit reference for the general statement that the sum of infinitely many WOT-continuous unitary reps is WOT-continuous?}

If $\pi:G\to\cU(\cH_\pi)$ is a unitary representation of a (topological) group, 
we denote by $\VN_\pi(G)$ the WOT-closed subalgebra of $\Bdd(\cH_\pi)$ generated by the subset $\pi(G)$.
Given $\xi,\eta\in\cH_\pi$ we denote by $\xi*_\pi\eta$ the function $g\mapsto \ip{\pi(g)\xi}{\eta}$.
If $\pi$ is furthermore WOT-continuous, then each function of the form $\xi*_\pi\eta$ belongs to $C_b(G)$, and the map $\xi\tp\overline{\eta} \mapsto \xi*_\pi\eta$ gives rise to a bounded linear map $\cH_\pi\ptp \overline{\cH_\pi}\to C_b(G)$.
 If, furthermore, we assume $G$ is locally compact, then we can appeal to the machinery of coefficient spaces as in e.g.~Arsac's thesis~\cite{Arsac_Lyon76}: in particular, $\VN_\pi(G)$ has a canonical predual $\FA_\pi(G)$, which is the image of the map $\cH_\pi\ptp\overline{\cH_\pi} \to C_b(G)$  equipped with the quotient norm.

\subsection{Helson sets and cb-Helson sets}
Let $E$ be a closed subset of a locally compact group $G$; let $\FA(G)$ denote the Fourier algebra of $G$,
as defined in \cite{Eym_BSMF64}.
 (Recall that if $G$ is locally compact and abelian, then $\FA(G)$ is naturally identified with the space of Fourier transforms of integrable functions on the dual group $\widehat{G}$.) Restriction of functions defines a contractive algebra homomorphism $\FA(G)\to C_0(E)$, whose kernel we denote by $I_G(E)$; let $\FA_G(E)$ be the quotient algebra $\FA(G)/I_G(E)$. There is a natural, injective homomorphism $\inc_{G,E}: \FA_G(E) \to C_0(E)$, which in general need not be surjective.

\begin{dfn}
We say that $E$ is a \dt{Helson subset of $G$} or a \dt{Helson set} for short, if $\inc_{G,E} :\FA(G)\to C_0(E)$ is surjective, hence an isomorphism of Banach spaces.
\end{dfn}

Equivalently: $E$ is a Helson subset of $G$ if and only if each continuous function on $E$ that vanishes at infinity can be extended to some function in the Fourier algebra of $G$.

Most of the operator space theory needed for this paper will be explained as and when we need it: all necessary background that is not explained here can be found in the standard references \cite{BleMer_book,ER_OSbook,Pis_OSbook}.
Let us recall briefly what is meant by the ``natural'' operator space structures on the spaces $C_0(G)$ and $\FA_G(E)$. Since $C_0(E)$ and $\VN(G)$ are $\Cst$-algebras, they have canonical operator space structures; $\VN(G)$ is a weak-star closed subspace of $\Bdd(L^2(G))$, hence its predual
%\footnotemark\
 $\FA(G)$ has an operator space structure, whose dual (as an operator space) coincides with the original operator space structure on $\VN(G)$; and finally, $\FA_G(E)$ is a quotient of $\FA(G)$, so it carries a natural quotient operator space structure. 
%\footnotetext{There is a slightly subtle point worth noting: a norm-closed subspace of $\Bdd(\ell_2)$ may have a Banach space predual, while as an operator space it need not be the operator space dual of any operator space. However, for \wstar-closed subspaces of $\Bdd(\ell_2)$ this problem does not arise. See \S1.4 and \S2.7 of \cite{BleMer_book} for further details.}

%\begin{rem}
%When we view $\FA(G)$ as an operator space in this way, we may equip $\FA(G)^*$ with a natural, dual operator space structure. Importantly for our purposes, this operator space structure coincides with the original operator space structure on $\VN(G)$.
%\end{rem}

\begin{dfn}\label{d:cb-Helson}
We say $E$ is a \dt{cb-Helson subset of $G$}, or a \dt{cb-Helson set} for short, if $\inc_{G,E}: \FA_G(E)\to C_0(E)$ is an isomorphism of operator spaces. 
\end{dfn}

\begin{rem}
Since the natural operator space structure on $C_0(E)$ is minimal, the map $\inc_{G,E}:\FA_G(E)\to C_0(E)$ is automatically completely bounded. (In fact, it is completely \emph{contractive}.) Thus $E$ is a cb-Helson set if and only if $\inc_{G,E}:\FA_G(E)\to C_0(E)$ is surjective and has completely bounded inverse.
\end{rem}

\begin{subsection}{Warning remarks on terminology}\label{ss:chips-vs-chips}
There seems to be some disagreement in the literature over the precise definition of Helson sets in non-compact groups. The original result of Helson that motivated the terminology ``Helson set'' was in the context of closed subsets of compact abelian groups. However, once one moves to locally compact abelian groups, terminology seems to differ. Some authors require that Helson subsets be compact as part of the definition; our convention, that they need not be compact, is in accordance with \cite{Herz_helson}.

There is also the notion of a \dt{Sidon subset} $E$ in a discrete abelian group $\Gamma$, which roughly speaking requires that functions on $E$ extend to elements of the Fourier--Stieltjes algebra ${\rm B}(\Gamma)$ with control of norm. It turns out that every Sidon subset of a discrete abelian group is automatically a Helson subset in our sense (see \cite[Theorem 5.7.3]{Rudin_FAG}). We prefer to keep the two concepts distinct.

It is actually easy to find examples of infinite cb-Sidon sets (where the definition is the obvious analogue of the cb-Helson condition). Let $\bbF_\infty$ denote the free group on a countably infinite number of generators, and let $E$ denote the set of generators: then the restriction map ${\rm B}(\bbF_\infty)\to \ell^\infty(E)$ is a complete quotient map, since it can be viewed as the adjoint of the known complete isometry $\MAX \ell^1(E) \hookrightarrow \Cst(\bbF_\infty)$.
\end{subsection}

\end{section}

\begin{section}{Calculations with cb-Helson constants}\label{s:cb-Helson constants}

Classically, a useful tool for studying Helson sets has been to quantify ``how Helson they are'', by associating a ``Helson constant'' to each closed subset of a given group. This has an obvious analogue for cb-Helson sets, as given by the following definition from \cite{CS_AGquot}.

\begin{dfn}[cb-Helson constants]
Let $E$ be a cb-Helson subset of a locally compact group~$G$. The \dt{cb-Helson constant of $E$}, which we denote by $\cbhel(E)$, is defined to be $\norm{\inc_{G,E}^{-1}}_{cb}$.
\end{dfn}

We adopt the convention that if $E$ is a closed subset of $G$ and not a cb-Helson subset, then $\cbhel(E)=+\infty$.

\begin{rem}\label{r:misc remarks}
 Let $E$ be a closed subset of $G$.
\begin{YCnum}
\item
By basic operator space theory, $\cbhel(E)=1$ if and only if the restriction map $\FA(G)\to C_0(E)$ is a complete quotient map of operator spaces.
\item\label{li:subsets}
If $F$ is a closed subset of $E$, then it follows from Tietze's extension theorem that $\cbhel(F)\leq\cbhel(E)$. (See \cite[Lemma 2.2]{CS_AGquot} for the details.)
\end{YCnum}
\end{rem}

Now, by Remark~\ref{r:misc remarks}\ref{li:subsets}, Theorem~\ref{t:amen-as-discrete} will follow from the following result on cb-Helson constants of finite subsets in certain groups.

\begin{prop}\label{p:bound for finite}
Let $G$ be a locally compact group such that $G_d$ is amenable. Let $F\subset G$ be a finite subset of size $n\geq 2$. Then $\cbhel(F) \geq n/ (2\sqrt{n-1})$.
\end{prop}

\begin{proof}
To ease notation, we abbreviate $\inc_{G,F}$ to $\inc_F$ and $\FA_G(F)$ to $\FA(F)$. We say ``identified'' as an abbreviation for ``identified completely isometrically as operator spaces''.

We identify $C(F)$ with $\MIN\ell_\infty^n$ and hence identify $C(F)^*$ with $\MAX\ell_1^n$. Since $\FA(G)\to\FA(F)$ is a complete quotient map (by definition), we may identify $\FA(F)^*$ with the closed linear span inside $\VN(G)$ of $\{\lambda_x \colon x\in F\}$.
Now consider $\Cst_\delta(G)$, the norm-closed subalgebra of $\Bdd(L_2(G))$ generated by the set of left translation operators $\{\lambda_x \colon x \in G\}$. Since $G_d$ is amenable, $\Cst(G_d)$ is nuclear \cite{Lance_nuc}; and since $\Cst(G_d)$ quotients onto $\Cst_\delta(G)$, it follows that $\Cst_\delta(G)$ is also nuclear. Thus $\FA(F)^*$ is identified with a closed subspace of a nuclear $\Cst$-algebra, and so (see \cite[Corollary 12.6]{Pis_OSbook})
\begin{equation}\label{eq:bodger}
 \inf_{V\subset \cK(\ell_2)} \dist_{cb}(\FA(F)^*, V) = 1
\end{equation}
where the infimum is taken over all $n$-dimensional subspaces of $\cK(\ell_2)$.

Let $V\subset\cK(\ell_2)$ be an $n$-dimensional subspace of $\cK(\ell_2)$.
Tnen by a result of Pisier (\cite[Theorem~7]{Pis_exactOS}),
\begin{equation}\label{eq:badger}
 \dist_{cb}(\MAX\ell_1^n, V) \geq \frac{n}{2\sqrt{n-1}}.
\end{equation}
By basic properties of cb-Banach--Mazur distance,
\[ \dist_{cb}(\MAX\ell_1^n,V) \leq \dist_{cb}(\MAX\ell_1^n,\FA(F)^*)\dist_{cb}(\FA(F)^*,V) .\]
Combining this with \eqref{eq:bodger} and \eqref{eq:badger} yields $n/(2\sqrt{n-1})\leq \dist_{cb}(\MAX\ell_1^n,\FA(F)^*)$. Since
\[ \dist_{cb}(\MAX\ell_1^n,\FA(F)^*) =\dist_{cb}(\MIN \ell_\infty^n, \FA(F)) \leq \norm{\inc_F}_{cb}\norm{\inc_F^{-1}}_{cb} = \cbhel(F) \]
the proof is complete.
\end{proof}

\begin{rem}
The theme of getting lower bounds on the Helson constants of finite sets is very standard in the study of the classical Helson condition. For instance, to prove that a closed non-trivial  arc in $\T$ is not Helson, one considers larger and larger finite subsets, each of which is an ``arithmetic progression'', and then shows that the Helson constant of a finite arithmetic progression in $\T$ tends to infinity as the number of points in the progression grows.
\end{rem}

\begin{rem}
The proof of Proposition~\ref{p:bound for finite}, and hence the proof of Theorem~\ref{t:amen-as-discrete}, works for any locally compact group $G$ such that $\Cst_\delta(G)$ is an exact $\Cst$-algebra. However, I do not know if this actually encompasses any new examples not covered by Theorem~\ref{t:amen-as-discrete}.
In fact: after this paper was submitted, I learned of the preprint~\cite{ruan-wiersma_preprint} by Ruan and Wiersma. As a special case of their results, they show that whenever $G$ is a locally compact amenable group for which $G_d$ is not amenable, such as ${\rm SO}(n,\Real)$ for all $n\geq 3$, the algebra $\Cst_\delta(G)$ is not even locally reflexive, hence {\it a priori} cannot be exact. See Example 4.1(3) and Theorem 4.3 in their paper for further details.
\end{rem}

\end{section}

\begin{section}{Theorem~\ref{t:realization}: a stronger version and some initial reductions}

Our goal in the rest of this paper is to prove the following result, which implies Theorem~\ref{t:realization}.

\begin{thm}[Embedding theorem, precise version]\label{t:embedding theorem}
Let $\Om$ be a compact Hausdorff space. Then there exists a family $(n(i))_{i\in\Ind}$ of positive integers such that, when we define $\bbG$ to be the group $\prod_{i\in\Ind} \bbU_{n(i)}$ with the product topology, there is a continuous injection from $\bj:\Om\to\bbG$ such that $\cbhel(\bj(\Om))=1$. %Moreover, we can take $\bbG$ to be an infinite product of finite-dimensional unitary groups.
\end{thm}

The proof of the theorem will be broken down into two propositions. Before giving all the details, let us provide some motivation.
We saw in the previous section that, in order to find large finite sets with small cb-Helson constant, we need to find maps $\MAX\ell_1^n \to \VN(G)$ which are not too far from being embeddings of operator spaces, and which send the standard basis vectors of $\ell_1^n$ to elements of $\lambda(G) \subset \ball(\VN(G))$.
Well, $\MAX\ell_1$ has a completely isometric embedding into $\Cst({\mathbb F}_\infty)$, which embeds as a $\Cst$-subalgebra of $\bprod_{n\geq 1} M_{2n}(\Cplx)$, which in turn embeds as a $\Cst$-sublagebra of $\VN(\SUTWO)$.
Although this embedding does not seem to send the standard basis of $\ell_1$ to elements of $\lm(\SUTWO)$, it does embed the standard basis of $\ell_1$ into $\prod_{n\geq } \bbU_{2n}$, and the latter is a compact group when given the product topology.
Since we want an embedding of $\Omega$ into $\lambda(\bbG)\subset\ball(\VN(\bbG))$ for some compact $\bbG$, the natural {\it Ansatz} is to take $\bbG=\prod_{n\geq 1} \bbU_{2n}$ and hope that the details work out.

Now we return to the task of proving Theorem~\ref{t:embedding theorem}. As we have already seen, a natural way to get complete quotient maps onto $C(\Om)$  is to dualize and look for completely isometric, \wswscts\ embeddings of $C(\Om)^*=M(\Om)$.

\begin{rem}
We already used, in a small way, the standard result that if $\sV$ is a minimal operator space then the dual operator space structure on $\sV^*$ coincides with the maximal operator space structure. However, for the proofs that follow, we don't actually need to know this.
\end{rem}

The first step is to observe that, in some sense, it suffices to get a good embedding of $M(\Om)$ into $\VN_\pi(\bbG)$ for a suitable WOT-continuous unitary representation $\pi:\bbG\to \cH_\pi$.

\begin{prop}\label{p:VN_pi not VN_lm}
Let $G$ be a compact group. Suppose there exist: a \emph{faithful}, WOT-continuous, unitary representation $\pi:G\to \cU(\cH_\pi)$;
 and a complete isometry
$J:M(\Om) \to \Bdd(\cH_\pi)$ which is \wswscts\ and maps $\{\delta_\om\colon\om\in\Om\}$ to a subset of $\pi(G)$.

For each $\om\in\Om$ let $\bj(\om)$ be the unique element of $G$ such that $J(\delta_\om)=\pi(\bj(\om))$.
Then:
\begin{YCnum}
\item $\bj:\Om\to G$ is a continuous injection with closed range;
\item\label{li:J_* is restriction}
 $J_*(h)(\om) = h(\bj(\om))$ for all $\om\in \Omega$;
\item $\cbhel(\bj(\Om))=1$.
\end{YCnum}
\end{prop}

\begin{proof}
Since $J$ is injective, so is $\bj$. Next we show that $\bj$ is continuous, which will imply that $\bj(\Om)$ is compact and hence closed in $G$.
Let $B$ denote $\ball(\Bdd(\cH_\pi))$ equipped with the relative WOT. Since $\pi:G\to B$ is a continuous map from a compact space to a Hausdorff one, it is a homeomorphism onto its range. Therefore to show $\bj:\Om\to G$ is continuous it suffices to show $\pi\circ\bj:\Om\to B$ is continuous. Let $\ev$ denote the map $\om\mapsto\delta_\om$, so that $\pi\circ\bj=J\circ\ev$. It is straightforward to check that when $M(\Om)$ is given the \wstar-topology, $\ev:\Om\to M(\Om)$ is continuous (just look at preimages of sub-basic open sets). Since $J:M(\Om)\to \Bdd(\cH_\pi)$ is \wswscts\ and contractive, it follows that $J\circ\ev:\Om \to (\ball(\cH_\pi), \wstar)$ is continuous. But the relative \wstar-topology and the relative WOT coincide on $\ball(\cH_\pi)$, so $J\circ\ev:\Om\to B$ is continuous as required.
 This completes the proof of part~(i).

Part (ii) is a straightforward consequence of the definitions of $\FA_\pi(G)$ and $\VN_\pi(G)$ and the pairing between them. In more detail:
% we have $\pair{ T }{\xi*_\pi \eta} = \ip{T\xi}{\eta}$
%for each $T\in\VN_\pi(G)$ and $\xi$,  $\eta\in\cH_\pi$\/.
%In particular, by the definition of $\xi*_\pi\eta$, we have
since
 $\pair{\pi(g)}{\xi*_\pi\eta} = (\xi*_\pi\eta)(g)$ for each $g\in G$,
% Therefore
%\[ 
$\pair{\pi(g)}{h}_{\VN_\pi-\FA_\pi} = h(g)$
% \quad\text{
for all $g\in G$ and all $h\in \FA_\pi(G)$.
%} \]
%\YCnote{Probably this is already in Arsac or Eymard}
Hence, by the definition of $\bj$,
\[ h(\bj(\om)) = \pair{J(\delta_\om)}{h}_{\VN_\pi-\FA_\pi} = \pair{\delta_\om}{J_*(h)}_{M(\Om)-C(\Om)} = J_*(h)(\om), \]
as required. Thus \ref{li:J_* is restriction} is proved.

Finally: let $R:\FA(G)\to C(\bj(\Om))$ be the completely contractive map given by restriction of functions, and let $\bj^*$ denote the completely isometric isomorphism $C(\bj(\Om))\to C(\Om)$ induced by~$\bj$.
Then by part~\ref{li:J_* is restriction},
\[ \bj^*R(h)(\om) = h(\bj(\om)) = J_*(h)(\om) \qquad(h\in\FA(G),\om\in\Om). \]
Let $\imath:\FA_\pi(G)\to \FA(G)$ be the natural inclusion map; this is well-defined and a complete isometry since $G$ is compact. Then $R\circ\imath=(\bj^*)^{-1}\circ J_*$, and $J_*:\FA_\pi(G)\to C(\Om)$ is a complete quotient map (since its adjoint $J$ is a complete isometry); a quick check now shows that $R$ is also a complete quotient map, as required. This finishes the proof of part (iii).
\end{proof}

\begin{prop}[An embedding in a product of matrix algebras]\label{p:an embedding of M(Omega)}
Let $\Om$ be a compact Hausdorff space. Then there exists a family $(\cH_i)_{i\in\Ind}$ of finite-dimensional Hilbert spaces, such that when we take $\cM\defeq\bprod_i \Bdd(\cH_i)$, there is a completely isometric, \wswscts\ embedding $J:(\min C(\Om))^* \to \cM$ such that $J(\delta_\om)\in \prod_i \cU(\cH_i)$ for all $\om\in\Om$.
\end{prop}

The proof of this proposition consists of modifying the construction in~\cite{Ble_sdual} of the ``standard dual of an operator space'', in the special case of $C(\Om)$. Since the details are straightforward but somewhat lengthy, the proof of this proposition will be deferred to the next section.

\begin{proof}[Proof of Theorem~\ref{t:embedding theorem}]
Let $(\cH_i)_{i\in\Ind}$, $\cM$ and $J$ be as provided by Proposition~\ref{p:an embedding of M(Omega)}.
Take $\bbG=\cU(\cM)=\prod_i \cU(\cH_i)$ and define $\bj:\Om\to\bbG$ by $J(\delta_\om)=\pi(\bj(\om))$.
Equip each $\cU(\cH_i)$ with the relative WOT, and equip $\bbG$ with the product topology. Since each $\cH_i$ is finite-dimensional, each $\cU(\cH_i)$ is compact, so $\bbG$ is compact.

Let $\cH=\hbox{$\ell^2$-}\bigoplus_{i\in\Ind}\cH_i$, and let $\theta: \cM\to\Bdd(\cH)$ be the natural, \wswscts\ inclusion. By abuse of notation we also use $\theta$ to denote the corresponding representation $\bbG\to \cU(\cH)$. We claim that: (i)~$\theta:\bbG\actson\cH$ is WOT-continuous; (ii)~$\VN_\theta(\bbG)=\theta(\cM)$. If~both of these hold, then the hypotheses of Proposition~\ref{p:VN_pi not VN_lm}
 are satisfied, and hence $\bj(\Om)$ is a cb-Helson subset of $\bbG$ with cb-Helson constant $1$.

To prove (i), for each $i\in\Ind$ let $\theta_i: \bbG \to \cU(\cH_i)$ be the cooordinate projection: this is  WOT-continuous by definition of the product topology. Hence, by Remark~\ref{r:product-of-WOT-cts}, the direct product $\prod_i\theta_i :\bbG \actson \cH$ is also WOT-continuous. It is easily checked that $\prod_i\theta_i$ coincides with~$\theta$, so we have proved~(i).

To prove (ii): note that the inclusion $\VN_\theta(\bbG)\subseteq\theta(\cM)$ is trivial. On the other hand, since each element of a unital $\Cst$-algebra $A$ is a linear combination of four elements of $\cU(A)$, we have
$\theta(\cM)\subseteq \operatorname{lin} \theta(\cU(\cM)) = \operatorname{lin}(\theta(\bbG)) \subseteq \VN_\theta(\bbG)$,
giving the converse inclusion.
(As pointed out by the referee: one could instead appeal to more general results on coefficient spaces of products of representations: see~e.g. \cite[Corollaire~3.13]{Arsac_Lyon76}.) 
 This completes the proof of~(ii), and hence completes the proof of the theorem.
\end{proof}

Theorem~\ref{t:embedding theorem} is a cb-analogue of a folklore result in the classical theory of Helson sets, which says that any compact space arises as a Helson subset of a product of (many!) copies of~$\T$. To emphasise the analogy, we state a more long-winded version of the classical result.

\begin{prop}[probably folklore]\label{p:univ-Helson}
Let $\Omega$ be a compact Hausdorff space. Let $\Gm=\bU(C(\Om))_d$, i.e.~the unitary group of $C(\Om)$ equipped with the \emph{discrete} topology, and let $G$ be the Pontryagin dual of $\Gm$, regarded as a subset of $\ball(\ell^\infty(\Gm))$.
Define $J:M(\Omega) \to \ell^\infty(\Gm)$
by $J(\mu)(\gm) = \int_\Om \gm\,d\mu$, and let $j(\om)=J(\delta_\om)$. Then
\begin{YCnum}
\item $j(\Om)$ is a compact subset of $G$;
\item $J:M(\Omega)\to \ell^\infty(\Gm)$ is \wswscts\ and bounded below;
\item if we identify $\FA(G)$ with $\ell^1(\Gm)$, then the restriction map $\FA(G)\to C(j(\Omega))$ is just $J_*$\/.
\end{YCnum}
In particular $j(\Omega)$ is a Helson subset of $G$.
\end{prop}

We omit the proof, which is straightforward book-keeping (for part (ii) we use the fact that any element of $\ball(C(\Om))$ can be written as $u_1+u_2+i(v_1+v_2)$ for some $u_1,u_2,v_1,v_2\in \Gm$; if we use \cite[Lemma 5.5.1]{Rudin_FAG} one can actually show that $J$ is an isometry). I do not know an exact reference for the precise statement above, but the construction and the method have appeared independently many times in the literature. For instance, as pointed out in \cite[\S5]{Herz_helson}, the group $G$ can be seen as the ``free compact abelian group'' generated by $\Om$, in an appropriate category-theoretic sense.
\end{section}

\begin{section}{The proof of Proposition~\ref{p:an embedding of M(Omega)}}

We seek a \wswscts\ and completely isometric embedding $J:(\MIN C(\Om))^*\to\cM$,
where $\cM$ is a product of matrix algebras,
 such that \begin{equation}\label{eq:desired tweak}
J(\delta_\om)\in\bU(\cM) \qquad\text{for all $\om\in\Omega$.}
\end{equation}

Let us temporarily ignore the requirement \eqref{eq:desired tweak}. Then there is a standard way to embed $(\MIN C(\Om))^*$ {\wswscts}ly and completely isometrically into a product of matrix algebras: this is a by-product of the definition of the \dt{standard dual} of a given operator space (see~\cite{Ble_sdual}).

With minor modifications this embedding procedure can be made to also satisfy~\eqref{eq:desired tweak}. Here are the details. Given an operator space $\sV$ and $p\in\Nat$, the ``standard norm'' on $M_p(\sV^*)$ is defined as follows:
given $\bmu=[\mu_{st}]\in M_p(\sV)$, let $T_\bmu: \sV\to \Mat_k$ be the linear map $v\mapsto [\mu_{st}(v)]$, and define $\norm{\bmu}_{(\rm SD)}$ to be $\norm{T_\bmu}_{cb}$.
To get \eqref{eq:desired tweak}, we note that if $\sV$ is a unital $\Cst$-algebra equipped with its canonical operator space structure, the cb-norm of $T_\bmu$ can be determined by testing on unitary matrices with entries in~$\sV$.
Although this follows from the Russo--Dye theorem, we shall give a more hands-on approach.

\begin{lem}[Determining $\norm{T_\bmu}_{cb}$]\label{l:dual norm using unitaries}
Let $\sV$ be a unital $\Cst$-algebra and let $\bmu\in M_p(\sV^*)$. Then
\[
\norm{\bmu}_{(\rm SD)} = \norm{T_\bmu}_{cb} = \sup_{n\in\Nat}\ \sup \{ \norm{(T_\bmu)_n (\bu) }_{M_n(\Mat_p)} \st \bu \in \cU(M_n(\sV))  \}.
\]
\end{lem}

\begin{proof}
We use the following well known trick: since $\sV$ is a unital $\Cst$-algebra, for each $a\in\ball(\sV)$ the block matrix
\[ U_a \defeq \twomat{(1-aa^*)^{1/2}}{a}{-a^*}{(1-a^*a)^{1/2}} \]
is well-defined and unitary in $M_2(\sV)$.

Let $n\in\Nat$ and let $E_{12}: M_{2n}(\sV) \to M_n(\sV)$ be compression to the (1,2) entry when we identify $M_{2n}(\sV)$ with $M_2(M_n(\sV))$.
By abuse of notation we also use $E_{12}$ to denote the corresponding compression map $M_{2n}(\Mat_p)\to M_n(\Mat_p)$.

Let $\bmu\in M_p(\sV^*)$. Since $E_{12}\circ (T_\bmu)_{2n} = (T_{\bmu})_n \circ E_{12}$, we have
\[ \begin{aligned}
\norm{(T_\bmu)_n}
 & = \sup\{ \norm{(T_\bmu)_n(\bv) }_{M_n(\Mat_p)} \st \bv\in \ball(M_n(\sV)) \} \\
 & = \sup\{ \norm{E_{12}(T_\bmu)_{2n} (U_{\bv}) }_{M_n(\Mat_p)} \st \bv\in \ball(M_n(\sV)) \} \\
 & \leq \sup\{ \norm{(T_\bmu)_{2n} (U_{\bv}) }_{M_{2n}(\Mat_p)} \st \bv\in \ball(M_n(\sV)) \} \\
 & \leq \sup\{ \norm{(T_\bmu)_{2n} (\bu) }_{M_{2n}(\Mat_p)} \st \bu \in \cU(M_{2n}(\sV))  \} \\
 & \leq \sup_{r\in\Nat} \sup\{ \norm{(T_\bmu)_r (\bu) }_{M_r(\Mat_p)} \st \bu \in \cU(M_r(\sV))  \}
 & \leq \norm{T_\bmu}_{cb}\,.
\end{aligned} \]
Taking the supremum over all $n$ on the left hand side, the result follows.
\end{proof}

Now we specialize to the case $\sV=C(\Om)$.
We can identify $M_n(C(\Om))$ (completely isometrically) with $C(\Om;\Mat_n)$, and under this identification $\cU(M_n(C(\Om)))$ is identified with $C(\Om; \bbU_n)$. To ease notation, we denote this unitary group by $\Gm_n$. Then, given $\bu =[u_{ij}] \in \Gm_n$, define $J_{n,\bu} : M(\Om) \to \Mat_n$ by
\begin{equation}
J_{n,\bu} (\mu) \defeq [ \mu(u_{ij}) ]
\end{equation}
If $\bmu = [\mu_{st}]\in M_p(M(\Om))$ then the ``canonical shuffle'' $M_p(\Mat_n) \iso M_n(\Mat_p)$ maps 
$(J_{n,\bu})_p (\bmu)$ to $(T_\bmu)_n(\bu)$, and so by Lemma~\ref{l:dual norm using unitaries} we get
\begin{equation}\label{eq:norm attained}
\sup_n \sup_{\bu\in \Gm_n} \norm{ (J_{n,\bu})_p(\bmu) }
= 
\sup_n \sup_{\bu\in \Gm_n} \norm{ (T_\bmu)_n(\bu) } = \norm{\bmu}_{(\rm SD)}\,.
\end{equation}

Let $\cM\defeq \bprod_{n\in \Nat} \bprod_{\bu\in \Gm_n} \Mat_n$
and define
$J : M(\Om) \longrightarrow \cM$
to be the direct product of the maps $J_{n,\bu}$. Equation~\eqref{eq:norm attained} implies that for each $p\in\Nat$ we have
\[ \norm{J_p(\bmu)} = \norm{\bmu}_{(\rm SD)} \quad\text{for all $\bmu\in M_p(M(\Om))$.} \]
so that $J:M(\Om)\to \cM$ is a complete isometry.
Note that for each $\om\in\Omega$,
$J_\bu(\delta_\om) = \bu(\om)\in \bbU_n$, hence $J(\delta_\om) \in \bU(\cM)$.
% To show $J$ is an $\Om$-$\bU(\cM)$ embedding, there is only one more thing left to check.
So to complete the proof of Proposition~\ref{p:an embedding of M(Omega)}, there is only one thing left to check.

\begin{lem}\label{l:J is normal}
$J$ is \wswscts.
\end{lem}

The analogous statement for the ``canonical embedding'' of a dual operator space can be found in \cite[Proposition 2.1]{Ble_sdual}, and a description of the pre-adjoint is given without proof after \cite[Proposition 3.1]{Ble_sdual}.
For sake of completeness, we give the details.

\begin{proof}[Proof of Lemma~\ref{l:J is normal}]
Let ${\bf T}_n$ denote the space of trace class operators on $\Cplx^n$, equipped with its canonical norm. Then
\[ \cM_* = \hbox{$\ell_1$-}\bigoplus_n \bigoplus_{\bu\in C(\Om;\bbU_n)} {\bf T}_n\,,\]
and
for any $\bS =(S_{n,\bu})\in \cM_*$ and $\mu\in M(\Om)$ we have
\begin{equation}\label{eq:one} 
 \pair{J^*(\bS)}{\mu}_{M(\Om)^*-M(\Om)}
 = \sum_{n\in\Nat} \sum_{\bu\in \Gm_n} \Tr(S_{n,\bu}J_{\bu}(\mu) ) \,.
\end{equation}
Define $h\in C(\Om)$ by
\[ h(\om) = \sum_{n\in\Nat}\sum_{\bu\in \Gm_n} \Tr\left[ S_{n,\bu} \bu(\om) \right] \qquad(\om\in\Om). \]
The sum is absolutely convergent, uniformly in $\om$, since $\sum_{n\in\Nat}\sum_{\bu\in\Gm_n} \norm{ S_{n,\bu} }_1 < \infty$ by our assumption on $\bS$. 
Therefore we have
\begin{equation}\label{eq:two}
 \mu(h) = \sum_{n\in\Nat}\sum_{\bu\in\Gm_n} 
\mu\left( \Tr\left[ S_{n,\bu}\bu(\cdot)\right]   \right)
 = \sum_{n\in\Nat} \sum_{\bu\in\Gm_n)} \Tr(S_{n,\bu}J_{\bu}(\mu) ) \,,
\end{equation}
and combining \eqref{eq:one} and \eqref{eq:two} gives
 $\pair{J^*(\bS)}{\mu} = \mu(h)$. Thus $J^*(\cM_*)\subseteq C(\Om) = M(\Om)_*$ as required.
(In fact, our argument constructs the pre-adjoint $J_*$ explicitly, as $J_*(\bS)\defeq h$.)
\end{proof}

\end{section}

\appendix
\begin{section}{Products of WOT-continuous representations}
The arguments that follow are surely not new. We include them because they highlight that local compactness plays no role in Remark~\ref{r:product-of-WOT-cts}.

Throughout $\Ind$ is a fixed indexing set, not necessarily countable.
Let $(\cH_i)_{i\in\Ind}$ be a family of Hilbert spaces, and let
$\cM\defeq \bprod_i \Bdd(\cH_i)$.
Of course, $\ball(\cM)=\prod_i \ball(\Bdd(\cH_i))$ as sets.
Form the $\ell^2$-direct sum $\cH=\ell^2\hbox{-}\bigoplus_{i\in\Ind} \cH_i$. If we regard $\cM$ as a von Neumann subalgebra of $\Bdd(\cH)$ in the natural way, via block-diagonal embedding, then we may equip $\ball(\cM)$ with the relative WOT inherited from $\Bdd(\cH)$, which will be denoted by~$\tau$. On the other hand, if we let $\tau_i$ be the relative WOT on $\ball(\Bdd(\cH_i)$ for each $i\in\Ind$, this gives us a product topology on $\prod_{i\in\Ind}\ball(\cH_i) = \ball(\cM)$.
% The next lemma tells us that these procedures yield the same topology on $\ball(\cM)$.

\begin{lem}[Products of unit balls in the WOT]
\label{l:product of relative WOT}
The identity map\hfill\break $\id:(\ball(\cM),\tau)\to \prod_{i\in\Ind} (\ball(\Bdd(\cH_i)), \tau_i)$ is a homeomorphism.
\end{lem}

\begin{proof}
We start by showing that $\id$ is continuous. By the universal property of product topologies, it suffices to show that for each $k\in\Ind$ the co-ordinate projection
$P_k: \ball(\cM) \to  \ball(\Bdd(\cH_k))$
is continuous for the respective weak operator topologies. But this follows immediately from the identity $P_k(T)\xi_k = T \imath_k(\xi_k)$, where $\imath_k:\cH_k\hookrightarrow\cH$ is the natural embedding of Hilbert spaces.

Note that
$\ball(\Bdd(\cH))$ is compact and Hausdorff in the relative WOT; and since $\cM$ is WOT -closed in $\Bdd(\cH)$, it is easily checked that $\ball(\cM)$ is WOT-closed in $\ball(\Bdd(\cH))$. Hence $(\ball(\cM),\tau)$ is compact and Hausdorff, so that $\id$, being a continuous surjection from a compact space to a Hausdorff space, is an open mapping.
\end{proof}

Now let $G$ be a topological group and for each $i\in\Ind$, let $\sigma_i:G_i\to\cU(\cH_i)$ be a WOT-continuous unitary representation.
Let $\sigma =\prod_i \sigma_i$ be the direct product representation $G \to \cU(\cH)$.
By the universal property of product topologies,
 $\sigma:G\to \prod_i (\ball(\Bdd(\cH_i)),\tau_i)$ is continuous. Applying Lemma~\ref{l:product of relative WOT}, we conclude that $\sigma$ is WOT-continuous as required.
%\YCnote{My original version of this proof used epsilons, nets and basic open neighbourhoods. The proof given above is more slick but perhaps less transparent.}
\end{section}

\subsection*{Acknowledgements}
The roots of this paper lie in discussions with E.~Samei. I would like to thank him and N.~Spronk for useful conversations and for their interest.
My thanks also go to the referee for useful feedback.

\bibliographystyle{siam}
\bibliography{cbhelson_bib2}

\end{document}